\documentclass[12pt]{article}
\usepackage{amsfonts}
\usepackage{amsmath, amssymb, amsthm}
\usepackage{enumerate}
\usepackage{epsfig}
\usepackage{verbatim}

\newtheorem{thm}{Theorem}
\newtheorem{lem}[thm]{Lemma}

\newtheorem{defn}[thm]{Definition}
\newtheorem{remark}[thm]{Remark}
\theoremstyle{remark}

\input{tcilatex}
\begin{document}

\title{On the finite element approximation of fourth order singularly
perturbed eigenvalue problems\\
}
\author{Hans-G\"{o}rg Roos \\
Institute of Numerical Mathematics\\
Technical University of Dresden\\
Dresden\\
GERMANY\\
\and Despo Savvidou and Christos Xenophontos\thanks{%
Corresponding author. Email: xenophontos@ucy.ac.cy} \\
Department of Mathematics and Statistics\\
University of Cyprus \\
P.O. BOX 20537\\
Nicosia 1678\\
CYPRUS}
\maketitle

\begin{abstract}
We consider fourth order singularly perturbed eigenvalue problems in
one-dimension and the approximation of their solution by the $h$ version of
the Finite Element Method (FEM). In particular, we use piecewise Hermite
polynomials of degree $p\geq 3$ defined on an {\emph{exponentially graded}}
mesh. We show that the method converges uniformly, with respect to the
singular perturbation parameter, at the optimal rate when the error in the
eigenvalues is measured in absolute value and the error in the eigenvectors
is measured in the energy norm. We also illustrate our theoretical findings
through numerical computations for the case $p=3$.
\end{abstract}



\section{Introduction}

\label{intro}

Singularly perturbed boundary value problems, and their numerical solution,
is a much studied topic in the last few decades (see the books \cite{mos}, 
\cite{morton}, \cite{rst} and the references therein). It is well known that
a main difficulty in the approximation to the solution of these problems is
the presence of \emph{boundary layers} in the solution. In order for the
approximate solution to be considered reliable, it must account for these
layers. In the context of the Finite Element Method (FEM), the robust
approximation of boundary layers requires either the use of the $h$ version
on non-uniform, {\emph{layer-adapted}} meshes (such as the Shishkin \cite%
{Shishkin2} or Bakhvalov \cite{B} mesh), or the use of the high order $p$
and $hp$ versions on specially designed (variable) meshes \cite{ss}. One
other layer-adapted mesh that has appeared in the literature is the \emph{%
exponentially graded} mesh (eXp) \cite{x0}. The finite element analysis on
this mesh appears in \cite{CX} for one-dimensional reaction-diffusion and
convection-diffusion problems, in \cite{xfl} for a two-dimensional
convection-diffusion problem posed in a square and in \cite{xfl1} for
two-dimensional reaction-diffusion problems posed in smooth domains. All the
aforementioned works concern second order singularly perturbed problems.
Only recently have fourth order singularly perturbed problems truly
attracted the attention of the numerical analysis research community (see,
e.g., \cite{CVX, FR, FR2, PZMX, X} for some recent results and \cite{Roos1,
RoosStynes, SunStynes} for some earlier results). In \cite{X} the finite
element analysis for a one-dimensional fourth order problem was carried out
on the eXp mesh, in the context of the $h$ version with piecewise
polynomials of degree $p\geq 3$. The purpose of this article is to extend
the results of \cite{X} to one-dimensional fourth order singularly perturbed 
\emph{eigenvalue} problems. To our knowledge, numerical analysis results for
such problems are scarce in the literature. The only relevant ones we could
find are the following: \cite{rod} in which a hybrid scheme based on
asymptotic expansions is employed in order to solve the thin hanging rod
problem and \cite{Roos} where the author presents a finite element
discretization of problem (\ref{de})--(\ref{bc}) (see ahead), using a
Shishkin mesh and polynomials of degree $p=3$. We will present a finite
element discretization using the eXp mesh and polynomials of degree $p\geq 3$%
, proving robust, optimal convergence in both the eigenvalues and the
eigenvectors, assuming they are simple. The error in the eigenvalues is
shown to decrease at the (expected) double rate, and the error in the
eigenvectors, measured in the energy norm, decreases at the optimal rate;
both do so independently of the singular perturbation parameter $\varepsilon 
$.

The rest of the paper is organized as follows: in Section \ref{model} we
present the model problem and its regularity. The discretization using the
exponentially graded mesh is presented in Section \ref{mesh} and in Section %
\ref{approx} we present our main results of parameter robust convergence in
the eigenvalues and the eigenvectors. Section \ref{nr} shows the results of
some numerical computations that illustrate the theoretical findings and in
Section \ref{concl} we give our conclusions.

With $I\subset \mathbb{R}$ a bounded open interval with boundary $\partial I$
and measure $|I|$, we will denote by $C^{k}(I)$ the space of continuous
functions on $I$ with continuous derivatives up to order $k$. We will use
the usual Sobolev spaces $H^{k}(I)=W^{k,2}(I)$ of functions on $I$ with $%
0,1,2,...,k$ generalized derivatives in $L^{2}(I)$, equipped with the norm
and seminorm $\left\Vert \cdot \right\Vert _{k,I}$ and $\left\vert \cdot
\right\vert _{k,I}$, respectively. We will also use the space 
\begin{equation*}
H_{0}^{k}\left( I\right) =\left\{ u\in H^{k}\left( I\right) :\left.
u^{(i)}\right\vert _{\partial I}=0,i=0,...,k-1\right\} .
\end{equation*}%
The norm of the space $L^{\infty }(I)$ of essentially bounded functions is
denoted by $\Vert \cdot \Vert _{\infty ,I}$. Finally, the notation
\textquotedblleft $a\lesssim b$\textquotedblright\ means \textquotedblleft $%
a\leq Cb$\textquotedblright\ with $C$ being a generic positive constant,
independent of any discretization or singular perturbation parameters and
possibly having different values in each occurrence -- dependence on various
other constants will be indicated.


\section{The model problem and its regularity\label{model}}

We consider the following \emph{eigenvalue problem}: Find $0\neq u(x)\in
C^{4}(I),\lambda \in \mathbb{C}$ such that 
\begin{equation}
\varepsilon ^{2}u^{(4)}(x)-\left( a(x)u^{\prime }(x)\right) ^{\prime
}+b(x)u(x)=\lambda u(x)\text{ in }I=(0,1),  \label{de}
\end{equation}%
along with the boundary conditions 
\begin{equation}
u(0)=u^{\prime }(0)=u^{\prime }(1)=u(1)=0.  \label{bc}
\end{equation}%
The parameter $0<\varepsilon \leq 1$ is given, as are the functions $a,b>0,$
which are assumed to be sufficiently smooth on the closed interval $%
\overline{I}=[0,1]$. Moreover, we assume that $\exists \;a_{0}\in \mathbb{R}$
such that%
\begin{equation*}
a(x)\geq a_{0}>0,b(x)\geq 0\;\forall x\in \overline{I}.
\end{equation*}%
It is easy to see that the problem (\ref{de}), (\ref{bc}) is self-adjoint.
As a result, the behavior of the eigenvalues is simplified, as $\varepsilon
\rightarrow 0$, as follows: for all positive eigenvalues $\lambda
_{k}(\varepsilon )$ there holds $\lim_{\varepsilon \rightarrow 0}\lambda
_{k}(\varepsilon )=\lambda _{k}(0)$. The values $\lambda _{k}(0)$ are the
eigenvalues of the reduced problem and if they are real then so are the $%
\lambda _{k}(\varepsilon )$. Moreover, $\lambda _{k}(\varepsilon )$ can be
expanded in a power series in $\varepsilon $ (see \cite{moser} for details).

The variational formulation of (\ref{de}), (\ref{bc}) reads: Find $0 \neq
u_{k}\in H_{0}^{2}\left( I\right) ,\lambda _{k}\in \mathbb{C}$ such that 
\begin{equation}
{\mathcal{B}}_{\varepsilon }\left( u_{k},v\right) =\lambda _{k}\left\langle
u_{k},v\right\rangle _{I}\;\;\forall \;v\in H_{0}^{2}\left( I\right) ,
\label{BuvFv}
\end{equation}%
where, with $\left\langle \cdot ,\cdot \right\rangle _{I}$ the usual $%
L^{2}(I)$ inner product, 
\begin{equation}
{\mathcal{B}}_{\varepsilon }\left( u,v\right) =\varepsilon ^{2}\left\langle
u^{\prime \prime },v^{\prime \prime }\right\rangle _{I}+\left\langle
au^{\prime },v^{\prime }\right\rangle _{I}+\left\langle bu,v\right\rangle
_{I}.  \label{Buv}
\end{equation}%
It follows that the bilinear form ${\mathcal{B}}_{\varepsilon }\left( \cdot
,\cdot \right) $ given by (\ref{Buv}) is coercive with respect to the \emph{%
energy norm} 
\begin{equation*}
\left\Vert u\right\Vert _{E,I}^{2}:=\varepsilon ^{2}\left\vert u\right\vert
_{2,I}^{2}+\left\Vert u\right\Vert _{1,I}^{2}\;,\;u\in H_{0}^{2}\left(
I\right) ,
\end{equation*}%
i.e., there exists $\gamma \in \mathbb{R}^{+}$, independent of $\varepsilon$%
, such that 
\begin{equation*}
{\mathcal{B}}_{\varepsilon }\left( u,u\right) \geq \gamma \left\Vert
u\right\Vert _{E,I}^{2}\;\;\forall \;u\in H_{0}^{2}\left( I\right) .
\end{equation*}%
The eigenfunctions $u_{k}$ are sufficiently smooth in $I$ and their first
derivative features \emph{boundary layers} at the endpoints. This is
described in the following result.

\begin{thm}
\label{thm_reg} Let $u\equiv u_{k}\in H_{0}^{2}(I)$ satisfy (\ref{BuvFv}).
Then%
\begin{equation*}
u=u_{S}+u_{BL}^{L}+u_{BL}^{R},
\end{equation*}%
and for $j=0,1,2,...$%
\begin{equation*}
|u_{S}^{(j)}(x)|\lesssim C_{j}(k),|(u_{BL}^{L})^{(j)}(x)|\lesssim \bar{C}%
_{j}(k)\varepsilon ^{1-j}e^{-\beta x/\varepsilon
},|(u_{BL}^{R})^{(j)}(x)|\lesssim \hat{C}_{j}(k)\varepsilon ^{1-j}e^{-\beta
(1-x)/\varepsilon },
\end{equation*}%
where $C_{j},\bar{C}_{j},\hat{C}_{j},\beta $ are positive constants
independent of $\varepsilon $.
\end{thm}

\begin{proof} In \cite{moser} we find the following decomposition for the
eigenfunctions: 
\begin{equation*}
u(x)=G_{0}(x,\varepsilon )+\varepsilon G_{1}(x,\varepsilon )\exp \left( -%
\frac{1}{\varepsilon }\int_{0}^{x}a^{1/2}(s)ds\right) +\varepsilon
G_{2}(x,\varepsilon )\exp \left( -\frac{1}{\varepsilon }%
\int_{x}^{1}a^{1/2}(s)ds\right)
\end{equation*}%
where $G_{i},i=0,1,2$ have asymptotic power series expansions with respect
to $\varepsilon $ (we omitted the dependence on $\lambda $.) The
decomposition and desired bounds follow from the above expression. \end{proof}

\begin{remark}
The dependence of the constants in the previous theorem, on $j$ and $k$ is
not explicitly known. Thus, if $C_{j}(k)\rightarrow \infty $ as $%
j\rightarrow \infty $ and/or $k\rightarrow \infty $, our results
deteriorate. Moreover, as our numerical results suggest, the computation of
higher modes becomes more difficult as $k$ is increased. This is in line
with classical results for non singularly perturbed eigenvalue problems,
see, e.g. \cite{Boffi}.
\end{remark}

\section{Discretization by an exponentially graded $h$-FEM\label{mesh}}

The discrete version of (\ref{BuvFv}) reads: Find $u_{k}^{h}\in V_{h}\subset
H_{0}^{2}\left( I\right) ,\lambda _{k}^{h}\in \mathbb{C}$ such that 
\begin{equation}
{\mathcal{B}}_{\varepsilon }\left( u_{k}^{h},v\right) =\lambda
_{k}^{h}\left\langle u_{k}^{h},v\right\rangle _{I}\;\;\forall \;v\in
V_{h}\subset H_{0}^{2}\left( I\right) ,  \label{BuvFvN}
\end{equation}%
with the finite dimensional subspace $V_{h}$ defined as follows: let 
\begin{equation*}
\Delta =\left\{ 0=x_{0}<x_{1}<...<x_{N}=1\right\}
\end{equation*}
be an arbitrary partition of $I$ and set 
\begin{equation*}
I_{j}=\left( x_{j-1},x_{j}\right) ,\quad h_{j}=x_{j}-x_{j-1},\quad j=1,...,N.
\end{equation*}%
With $\mathbb{P}_{p}(\alpha ,\beta )$ the space of polynomials of degree
less than or equal to $p\geq 2N+1$ on the interval $(\alpha ,\beta )$, we
define the subspace $V_{h}\subset H_{0}^{2}(I)$ as 
\begin{equation}
V_{h}=\left\{ u\in H_{0}^{2}\left( I\right) :u|_{I_{j}}\in \mathbb{P}%
_{p}\left( I_{j}\right) ,j=1,...,N\right\} .  \label{Vh}
\end{equation}%
We note that the space $V_{h}$ consists of the classical (piecewise) Hermite
polynomials (see, e.g., \cite{AI}), hence we quote the following relevant
results.

\begin{defn}
\cite{AI} \label{defn_hermite} Let $\{x_{i}\}_{i=0}^{N}$ be an arbitrary
partition of the interval $[a,b]$ and suppose that for a sufficiently smooth
function $f(x),x\in \lbrack a,b]$, the values 
\begin{equation*}
f(x_{i})=y_{i}\in \mathbb{R}\;,\;f^{\prime }(x_{i})=y_{i}^{\prime }\in 
\mathbb{R}\;,\;i=0,1,...,N
\end{equation*}%
are given. Then there exists a unique polynomial $f^{I}\in \mathbb{P}%
_{2N+1}\left( a,b\right) $, called the {\emph{Hermite}} interpolant of $f$,
given by 
\begin{equation*}
f^{I}(x)=\sum_{i=0}^{N}\left( y_{i}H_{0,i}(x)+y_{i}^{\prime
}H_{1,i}(x)\right) ,
\end{equation*}%
where, with $L_{i}(x)$ the Lagrange polynomial of degree $N$ associated with
node $x_{i}$, 
\begin{equation*}
H_{0,i}(x)=[1-2(x-x_{i})\frac{dL_{i}}{dx}(x_{i})]L_{i}^{2}(x)\;,%
\;H_{1,i}(x)=(x-x_{i})L_{i}^{2}(x).
\end{equation*}
\end{defn}

\begin{thm}
\cite[Thm 1.12]{AI}\label{AI} Let $v\in C^{2n+2}\left( [a,b]\right) $ and
let $\Delta =\left\{ x_{i}\right\} _{i=0}^{N}$ be a mesh on $[a,b]$ with
maximum meshsize $h$ and with $N$ a multiple of $n$. If $v^{I}$ is the
piecewise Hermite interpolant of $v$ from Definition \ref{defn_hermite},
having degree at most $2n+1$ on each subinterval $[x_{i-1},x_{i}],i=1,...,N$%
, then 
\begin{equation*}
\left\Vert v^{(\ell )}-(v^{I})^{(\ell )}\right\Vert _{\infty ,I}\lesssim
h^{2n+2-\ell }\left\Vert v^{(2n+2)}\right\Vert _{\infty ,I}\;,\ell
=0,1,...,2n+1.
\end{equation*}
\end{thm}

We mention in passing that the classical theory of eigenvalue problems (see,
e.g., \cite{SF}) gives, in the case when $\varepsilon $ is \emph{fixed} and
piecewise cubic polynomials are used on a uniform mesh with meshsize $h$, 
\begin{equation*}
\lambda _{k}\leq \lambda _{k}^{h}\leq \lambda _{k}+C(\varepsilon )\lambda
_{k}^{2}h^{4},
\end{equation*}%
with $h\leq h_{0}(\varepsilon )$ for some $h_{0}$. \ Numerical experiments,
however, indicate that this estimate does not hold uniformly with respect to 
$\varepsilon $. This is due to the boundary layer components that are
present in the (first derivative of the) eigenfunctions and in view of
Theorem \ref{thm_reg}, the `challenge' lies in approximating the
one-dimensional boundary layer function 
\begin{equation}
e^{-\beta x/\varepsilon },\beta \in \mathbb{R}^{+},x\in \lbrack
0,1],\varepsilon \in (0,1].  \label{typical}
\end{equation}

As mentioned before, there are several layer adapted meshes in the
literature, perhaps the most widely known being the Shishkin or S-type
meshes. In this article we choose to use the exponentially graded (eXp) mesh
from \cite{x0} -- therein the mesh appears for the first time in the
literature. (See also \cite{FX} for a connection between the eXp mesh and
S-type meshes.) To define the mesh, let the mesh points be chosen as
follows: with $N>4$ a multiple of $4$, we split the interval $[0,1]$ into 
\begin{equation*}
\lbrack 0,x_{N/4-1}]\text{ , }[x_{N/4-1},x_{3N/4+1}]\text{ , }[x_{3N/4+1},1]
\end{equation*}%
and on $[x_{N/4-1},x_{3N/4+1}]$ we choose an equidistant mesh with $N/4+1$
elements. For the other two subintervals the mesh will be \emph{%
exponentially graded} with $N/4-1$ elements. In particular, the mesh is
given by a continuous, monotonically increasing, piecewise continuously
differentiable, generating function $\phi $ with $\phi (0)=0$. Then, the
nodal points in our mesh are given by 
\begin{equation}
x_{j}=\left\{ 
\begin{array}{ccc}
\frac{\varepsilon }{\beta }(p+1)\phi \left( \frac{j}{N}\right) & , & 
j=0,1,...,N/4-1 \\ 
x_{N/4-1}+\left( \frac{x_{3N/4}-x_{N/4-1}}{N/2+2}\right) \left( j-\frac{N}{4}%
+1\right) & , & j=N/4,...,3N/4 \\ 
1-\frac{\varepsilon }{\beta }(p+1)\phi \left( \frac{N-j}{N}\right) & , & 
j=3N/4+1,...,N%
\end{array}%
\right.  \label{nodes}
\end{equation}%
with 
\begin{equation}
\phi (t)=-\ln \left[ 1-4C_{p,\varepsilon }t\right] ,\;t\in \lbrack
0,1/4-1/N],  \label{phi}
\end{equation}%
where 
\begin{equation}
C_{p,\varepsilon }=1-\exp \left( -\frac{\beta }{(p+1)\varepsilon }\right)
\in \mathbb{R}^{+}.  \label{Cpe}
\end{equation}

An example of this mesh is shown in Figure \ref{expmesh2}.

\setlength{\unitlength}{2mm} 
\begin{figure}[h]
\begin{center}
\begin{picture}(64,10)

\put(0,5){\line(1,0){64}}
{\linethickness{0.5mm}
\put(0,4){\line(0,1){2}}
}
\put(-0.2,2.2){$0$}
\put(0.6,4.5){\line(0,1){1}}
\put(2,4.5){\line(0,1){1}}
\put(5,4.5){\line(0,1){1}}
\put(9,4.5){\line(0,1){1}}
\put(11,3){$\cdots$}
{\linethickness{0.5mm}
\put(18,4.5){\line(0,1){1}}
}
\put(16,3){$x_{N/4-1}$}
\put(32,3){$\cdots$}
\multiput(22,4.5)(4.2,0){6}{\line(0,1){1}}
{\linethickness{0.5mm}
\put(48,4.5){\line(0,1){1}}
}
\put(45,3){$x_{3N/4+1}$}
\put(57,4.5){\line(0,1){1}}
\put(60,4.5){\line(0,1){1}}
\put(62,4.5){\line(0,1){1}}
\put(63.25,4.5){\line(0,1){1}}
\put(53,3){$\cdots$}
{\linethickness{0.5mm}
\put(64,4){\line(0,1){2}}
}
\put(63.75,2.2){$1$}
\end{picture}
\end{center}
\caption{Example of the exponential mesh.}
\label{expmesh2}
\end{figure}
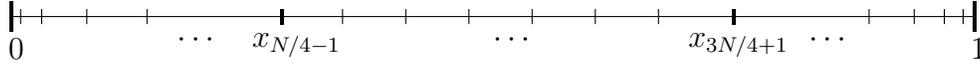

{We also define the function $\psi $ by $\phi =-\ln \psi $, which gives $%
\psi (t)=1-2C_{p,\varepsilon }t$ as well as $\psi ^{\prime
}(t)=-2C_{p,\varepsilon }\in \mathbb{R}^{-}$. } The meshwidth $h_{j}$ in the
intervals $[0,x_{N/4-1}],[x_{3N/4+1},1]$ satisfies \cite{CX}, 
\begin{equation}
h_{j}\leq \frac{\varepsilon }{\beta }(p+1)N^{-1}\max_{I_{j}}\phi ^{\prime
}\leq \frac{\varepsilon }{\beta }(p+1)e^{\frac{x_{j}}{(p+1)\varepsilon }%
},j=1,...,N/4-1,3N/4+1,...,N.  \label{hj0}
\end{equation}%
Moreover, under the assumption $\frac{\varepsilon }{\beta }(p+1)\ln (N-4)<1$%
, which means that $\varepsilon $ is small and we are in the singularly
perturbed case, it was shown in \cite{CX} that 
\begin{equation}
e^{-\beta x_{N/4-1}/\varepsilon }+e^{-(1-\beta x_{3N/4+1})/\varepsilon
}\lesssim N^{-(p+1)}.  \label{tp}
\end{equation}%
The interpolation result below (Lemma \ref{lemma_interp}) was established in 
\cite{X} under the the (stronger, but common) assumption 
\begin{equation}
\varepsilon <N^{-1}.  \label{Ass1}
\end{equation}%
(This is needed in order to be able to approximate the smooth part of the
solution at the correct rate.) Note that under this assumption, one has $%
h_{j}\lesssim N^{-1}$ for all $I_{j}\subset I$ and the problem \emph{is}
singularly perturbed.

\section{Error estimates}

\label{approx}

We begin by noting that in our setting, Theorem \ref{AI} gives%
\begin{equation}
\left\Vert v^{(k)}-(v^{I})^{(k)}\right\Vert _{\infty ,I_{j}}\lesssim
h_{j}^{p+1-k}\left\Vert v^{(p+1)}\right\Vert _{\infty
,I}\;,k=0,1,...,p,j=1,...,N.  \label{Hinterp}
\end{equation}%
Using the above and the definition of the exponential mesh the following
lemma was established in \cite{X}.

\begin{lem}
\label{lemma_interp}Let $u_{BL}$ be given by (\ref{typical}) and let $%
u_{BL}^{I}\in V_{h}$ be its interpolant as in Theorem \ref{AI} \ based on
the mesh $\Delta =\{x_{j}\}_{j=1}^{N}$ with nodes (\ref{nodes}) obtained
with the mesh generating function $\phi $ given by (\ref{phi}). Then 
\begin{equation}
\left\Vert \left( u_{BL}-u_{BL}^{I}\right) ^{(\ell )}\right\Vert _{\infty
,I}\lesssim \varepsilon ^{1-k}N^{-(p+1-\ell )}\text{ },\ell =0,1,...,p,
\label{Linf}
\end{equation}%
and 
\begin{equation}
\left\vert u_{BL}-u_{BL}^{I}\right\vert _{2,I}\lesssim \varepsilon
^{-1/2}N^{-p+1}\text{.}  \label{H2}
\end{equation}
\end{lem}

The above lemma allows us to prove the following

\begin{lem}
\label{lemma_interp2}Let $u$ be the solution of (\ref{BuvFv}) and let $%
u^{I}\in V_{h}$ be its interpolant as in Theorem \ref{AI} \ based on the
mesh $\Delta =\{x_{j}\}_{j=1}^{N}$ with nodes (\ref{nodes}) obtained with
the mesh generating function $\phi $ given by (\ref{phi}). Then 
\begin{equation*}
\left\Vert \left( u-u^{I}\right) ^{(\ell )}\right\Vert _{\infty ,I}\lesssim
\varepsilon ^{1-\ell }N^{-(p+1-\ell )}\text{ },\ell =0,1,...,p,
\end{equation*}%
and 
\begin{equation*}
\left\vert u-u^{I}\right\vert _{2,I}\lesssim \varepsilon ^{-1/2}N^{-p+1},
\end{equation*}%
hence%
\begin{equation*}
\left\Vert u-u^{I}\right\Vert _{E,I}\lesssim N^{-p+1}.
\end{equation*}
\end{lem}

\begin{proof}

We use the decomposition of Theorem \ref{thm_reg}, $%
u=u_{S}+u_{BL}^{L}+u_{BL}^{R},$ and denote the interpolant by $%
u^{I}=u_{S}^{I}+\left( u_{BL}^{L}\right) ^{I}+\left( u_{BL}^{R}\right) ^{I},$
with the obvious notation. Then, 
\begin{equation*}
\left\Vert \left( u-u^{I}\right) ^{(\ell )}\right\Vert _{\infty ,I}\lesssim
\left\Vert \left( u_{S}-u_{S}^{I}\right) ^{(\ell )}\right\Vert _{\infty
,I}+\left\Vert \left( u_{BL}^{L}-\left( u_{BL}^{L}\right) ^{I}\right)
^{(\ell )}\right\Vert _{\infty ,I}+\left\Vert \left( u_{BL}^{R}-\left(
u_{BL}^{R}\right) ^{I}\right) ^{(\ell )}\right\Vert _{\infty ,I},
\end{equation*}%
with the last two terms being handled by Lemma \ref{lemma_interp} and the
first one by standard techniques and (\ref{Ass1}). The other estimates are
shown in a similar fashion. \end{proof}

Returning to the eigenvalue problem, set%
\begin{equation*}
E_{k}=span\left\{ u_{k}\right\} \;,\;E_{k}^{h}=span\left\{ u_{k}^{h}\right\}
\;,\;E^{h}=\oplus _{i=1}^{k}E_{i}^{h}.
\end{equation*}%
Then, the discrete min-max condition says (see \cite[eq. (7.6)]{Boffi})%
\begin{equation}
\lambda _{k}^{h}=\min_{E^{h}\in V_{h}^{(k)}}\max_{v\in E^{h}}\frac{{\mathcal{%
B}}_{\varepsilon }(v,v)}{\left\langle v,v\right\rangle _{I}},  \label{minmax}
\end{equation}%
where $V_{h}^{(k)}$ denotes the set of all subspaces of $V_{h}$ with
dimension $k$. We choose%
\begin{equation}
E^{h}=\Pi _{h}V^{(k)}  \label{Eh}
\end{equation}%
in (\ref{minmax}), where%
\begin{equation*}
V^{(k)}=\oplus _{i=1}^{k}E_{i}
\end{equation*}%
and $\Pi _{h}:V\rightarrow V_{h}$ is the Ritz projection, defined by 
\begin{equation}
{\mathcal{B}}_{\varepsilon }\left( u-\Pi _{h}u,v\right) =0\;\forall \;v\in
V_{h}.  \label{proj}
\end{equation}%
We may do so since, for $h$ sufficiently small, the bound%
\begin{equation}
\left\Vert \Pi _{h}v\right\Vert _{E,I}\geq \left\Vert v\right\Vert
_{E,I}-\left\Vert v-\Pi _{h}v\right\Vert _{E,I}\;\forall \;v\in V
\label{bound}
\end{equation}%
ensures that the dimension of $E^{h}$ is equal to $k$. In particular, if we
take $h$ such that%
\begin{equation*}
\left\Vert v-\Pi _{h}v\right\Vert _{E,I}\leq \frac{1}{2}\left\Vert
v\right\Vert _{E,I}\;\forall \;v\in V^{(k)},
\end{equation*}%
then $\Pi _{h}$ is injective from $V^{(k)}$ to $E^{h}$. (The smallness of $h$
depends on $k$). See \cite{Boffi} for more details.

As in \cite{SunStynes}, we have%
\begin{eqnarray*}
\left\Vert u-\Pi _{h}u\right\Vert _{E,I}^{2} &=&{\mathcal{B}}_{\varepsilon
}\left( u-\Pi _{h}u,u-\Pi _{h}u\right) ={\mathcal{B}}_{\varepsilon }\left(
u-\Pi _{h}u,u-\Pi _{h}u-v\right) \\
&=&{\mathcal{B}}_{\varepsilon }\left( u-\Pi _{h}u,u-\widetilde{v}\right) \\
&\lesssim &\left\Vert u-\Pi _{h}u\right\Vert _{E,I}\left\Vert u-\widetilde{v}%
\right\Vert _{E,I}
\end{eqnarray*}%
with $\widetilde{v}=\Pi _{h}u-v\in V_{h}$ arbitrary. Hence, with $u^{I}$ the 
$p^{th}$ degree interpolant of $u$ on the exponential mesh, we have by Lemma %
\ref{lemma_interp2}%
\begin{equation}
\left\Vert u-\Pi _{h}u\right\Vert _{E,I}\lesssim \left\Vert u-\widetilde{v}%
\right\Vert _{E,I}\lesssim \left\Vert u-u^{I}\right\Vert _{E,I}\lesssim
N^{-p+1}\lesssim h^{p-1}.  \label{projection}
\end{equation}%
The above will be utilized in establishing the following result for the
approximation of the eigenvalues.

\begin{thm}
\label{thm_evals}Let $\lambda _{k},u_{k}$ be the solution of (\ref{BuvFv})
and $\lambda _{k}^{h},u_{k}^{h}$ the solution of (\ref{BuvFvN}) on the eXp
mesh. Assuming $\left\langle u_{k},u_{k}\right\rangle _{I}=1=\left\langle
u_{k}^{h},u_{k}^{h}\right\rangle _{I}$ as well as $\left\langle
u_{k},u_{k}^{h}\right\rangle _{I}>0$, we have for all $h\leq h_{0},$ with $%
h_{0}$ independent of $\varepsilon $, the bound 
\begin{equation*}
\lambda _{k}\leq \lambda _{k}^{h}\lesssim \bar{C}(k)\lambda _{k}\left(
1+h^{2p-2}\right) ,
\end{equation*}%
with $\bar{C}(k)$ independent of $\varepsilon $.
\end{thm}

\begin{proof} The proof follows \cite[Sec. 2.8]{Boffi} and \cite[Ch. 6]{SF}.
Let $k$ be fixed. Using (\ref{Eh}) in (\ref{minmax}) gives%
\begin{equation*}
\lambda _{k}^{h}\leq \max_{w\in E^{h}}\frac{{\mathcal{B}}_{\varepsilon }(w,w)%
}{\left\langle w,w\right\rangle _{I}}=\max_{v\in V^{(k)}}\frac{{\mathcal{B}}%
_{\varepsilon }(\Pi _{h}v,\Pi _{h}v)}{\left\langle \Pi _{h}v,\Pi
_{h}v\right\rangle _{I}}.
\end{equation*}%
Note that%
\begin{equation*}
{\mathcal{B}}_{\varepsilon }\left( \Pi _{h}v,\Pi _{h}v\right) ={\mathcal{B}}%
_{\varepsilon }\left( v,v\right) +2{\mathcal{B}}_{\varepsilon }\left( \Pi
_{h}v,\Pi _{h}v-v\right) -{\mathcal{B}}_{\varepsilon }\left( \Pi _{h}v-v,\Pi
_{h}v-v\right)
\end{equation*}%
with the last term positive and the second to last zero. Thus,%
\begin{equation*}
{\mathcal{B}}_{\varepsilon }\left( \Pi _{h}v,\Pi _{h}v\right) \leq {\mathcal{%
B}}_{\varepsilon }\left( v,v\right) .
\end{equation*}%
Writing%
\begin{equation*}
v=\sum_{i=1}^{k}c_{i}u_{i}\ ,c_{i}\in \mathbb{R},
\end{equation*}%
we have%
\begin{eqnarray*}
{\mathcal{B}}_{\varepsilon }\left( \Pi _{h}v,\Pi _{h}v\right) &\leq &{%
\mathcal{B}}_{\varepsilon }\left(
\sum_{i=1}^{k}c_{i}u_{i},\sum_{j=1}^{k}c_{j}u_{j}\right)
=\sum_{i=1}^{k}c_{i}^{2}{\mathcal{B}}_{\varepsilon }\left(
u_{i},u_{i}\right) =\sum_{i=1}^{k}c_{i}^{2}\lambda _{i}\left\langle
u_{i},u_{i}\right\rangle _{I} \\
&\leq &\sum_{i=1}^{k}c_{i}^{2}\lambda _{i}\leq C(k)\lambda _{k}
\end{eqnarray*}%
and thus,%
\begin{equation*}
\lambda _{k}^{h}\leq C(k)\lambda _{k}\max_{v\in V^{(k)}}\frac{1}{\left\Vert
\Pi _{h}v\right\Vert _{0,I}^{2}}.
\end{equation*}%
Note that%
\begin{equation*}
\left\Vert v\right\Vert _{0,I}^{2}=\left\langle v,v\right\rangle
_{I}=\left\langle
\sum_{i=1}^{k}c_{i}u_{i},\sum_{j=1}^{k}c_{j}u_{j}\right\rangle
_{I}=\sum_{i=1}^{k}c_{i}^{2}\left\langle u_{i},u_{i}\right\rangle
_{I}=\sum_{i=1}^{k}c_{i}^{2}=C(k).
\end{equation*}%
Moreover,%
\begin{equation*}
\left\Vert v-\Pi _{h}v\right\Vert _{0,I}^{2}=\left\langle v-\Pi _{h}v,v-\Pi
_{h}v\right\rangle _{I}=\left\Vert v\right\Vert _{0,I}^{2}-2\langle v,\Pi
_{h}v\rangle _{I}+\left\Vert \Pi _{h}v\right\Vert _{0,I}^{2},
\end{equation*}%
hence,%
\begin{equation*}
\left\Vert \Pi _{h}v\right\Vert _{0,I}^{2}=\left\Vert v-\Pi _{h}v\right\Vert
_{0,I}^{2}-C(k)+2\langle v,\Pi _{h}v\rangle _{I}.
\end{equation*}%
The term $\left\Vert v-\Pi _{h}v\right\Vert _{0,I}^{2}$ may be handled by
Lemma \ref{lemma_interp2}. For the term $\langle v,\Pi _{h}v\rangle _{I}$,
we have%
\begin{eqnarray*}
\left\vert \langle v,\Pi _{h}v\rangle _{I}\right\vert &=&\left\vert
\sum_{i=1}^{k}c_{i}\left\langle u_{i},\Pi _{h}v\right\rangle _{I}\right\vert
\leq \sum_{i=1}^{k}\left\vert c_{i}\right\vert \left\vert \left\langle
u_{i},\Pi _{h}v\right\rangle _{I}\right\vert \leq \sum_{i=1}^{k}\left\vert
c_{i}\right\vert \left\vert \lambda _{i}^{-1}{\mathcal{B}}_{\varepsilon
}\left( u_{i},\Pi _{h}v\right) \right\vert \\
&\leq &\sum_{i=1}^{k}\left\vert c_{i}\right\vert \left\vert \lambda
_{i}^{-1}\right\vert \left\vert {\mathcal{B}}_{\varepsilon }\left( u_{i}-\Pi
_{h}u_{i},v-\Pi _{h}v\right) \right\vert \leq \sum_{i=1}^{k}\left\vert
c_{i}\right\vert \left\vert \lambda _{i}^{-1}\right\vert \left\Vert
u_{i}-\Pi _{h}u_{i}\right\Vert _{E,I}\left\Vert v-\Pi _{h}v\right\Vert _{E,I}
\\
&\leq &\sum_{i=1}^{k}\left\vert \frac{c_{i}}{\lambda _{i}}\right\vert
h^{2p-2}\lesssim \left[ \sum_{i=1}^{k}\frac{c_{i}^{2}}{\lambda _{i}^{2}}%
\right] ^{1/2}h^{2p-2}=\tilde{C}(k)h^{2p-2},
\end{eqnarray*}%
where Galerkin orthogonality and the coercivity of the bilinear form were
used. Since%
\begin{equation*}
\left\Vert \Pi _{h}v\right\Vert _{0,I}^{2}\geq \max_{v\in V^{(k)}}\left\vert
2\langle v,\Pi _{h}v\rangle _{I}+\left\Vert v-\Pi _{h}v\right\Vert
_{0,I}^{2}\right\vert -C(k).
\end{equation*}%
we obtain%
\begin{equation*}
\left\Vert \Pi _{h}v\right\Vert _{0,I}^{2}\gtrsim \hat{C}(k)\left(
h^{2p-2}-1\right) ,
\end{equation*}%
\bigskip with $\hat{C}(k)=\min \left\{ 1,C(k),\tilde{C}(k)\right\} $. This
gives%
\begin{equation*}
\lambda _{k}^{h}\leq C(k)\lambda _{k}\frac{1}{\hat{C}(k)\left(
h^{2p-2}-1\right) }\lesssim \bar{C}(k)\lambda _{k}\left( 1+2h^{2p-2}\right) ,
\end{equation*}%
as desired. \end{proof}

For the approximation of the eigenfunctions, we have the following result,
under the assumption that all eigenvalues are distinct.

\begin{thm}
\label{thm_evecs}Let $\lambda _{k},u_{k}$ be the solution of (\ref{BuvFv})
and $\lambda _{k}^{h},u_{k}^{h}$ the solution of (\ref{BuvFvN}) on the eXp
mesh. Assume that $\left\langle u_{k},u_{k}\right\rangle _{I}=1=\left\langle
u_{k}^{h},u_{k}^{h}\right\rangle _{I}$ , $\left\langle
u_{k},u_{k}^{h}\right\rangle _{I}>0$ and that all eigenvalues are distinct.
Then, 
\begin{equation*}
\Vert u_{k}-u_{k}^{h}\Vert _{E,I}\lesssim C(k)h^{p-1},
\end{equation*}%
with $C(k)\in \mathbb{R}$ independent of $\varepsilon ,u$ and $p$.
\end{thm}

\begin{proof} We again follow \cite{Boffi} (see also \cite{SF}), and introduce
the following quantity:%
\begin{equation*}
\rho _{k}^{h}=\max_{k\neq j}\frac{\left\vert \lambda _{k}\right\vert }{%
|\lambda _{k}-\lambda _{j}^{h}|}.
\end{equation*}%
We also consider the $L^{2}$ projection of $\Pi _{h}u_{k}$ onto $%
span\{u_{k}^{h}\}$,%
\begin{equation}
w_{k}^{h}=\left\langle \Pi _{h}u_{k},u_{k}^{h}\right\rangle _{I}u_{k}^{h},
\label{wk}
\end{equation}%
which we use as follows:%
\begin{equation}
\Vert u_{k}-u_{k}^{h}\Vert _{0,I}\leq \Vert u_{k}-\Pi _{h}u_{k}\Vert
_{0,I}+\Vert \Pi _{h}u_{k}-w_{k}^{h}\Vert _{0,I}+\Vert
w_{k}^{h}-u_{k}^{h}\Vert _{0,I}  \label{8.6}
\end{equation}%
The first term in (\ref{8.6}) is estimated using Lemma \ref{lemma_interp2}.
To deal with the second term, note that%
\begin{equation*}
\Pi _{h}u_{k}-w_{k}^{h}=\sum_{j\neq k}\left\langle \Pi
_{h}u_{k},u_{j}^{h}\right\rangle _{I}u_{j}^{h},
\end{equation*}%
which gives 
\begin{equation}
\left\Vert \Pi _{h}u_{k}-w_{k}^{h}\right\Vert _{0,I}^{2}=\sum_{j\neq
k}\left\langle \Pi _{h}u_{k},u_{j}^{h}\right\rangle _{I}^{2}.  \label{8.7}
\end{equation}%
We have%
\begin{equation*}
\left\langle \Pi _{h}u_{k},u_{j}^{h}\right\rangle _{I}=\frac{1}{\lambda
_{j}^{h}}\mathcal{B}_{\varepsilon }\left( \Pi _{h}u_{k},u_{j}^{h}\right) =%
\frac{1}{\lambda _{j}^{h}}\mathcal{B}_{\varepsilon }\left(
u_{k},u_{j}^{h}\right) =\frac{\lambda _{k}}{\lambda _{j}^{h}}\left\langle
u_{k},u_{j}^{h}\right\rangle _{I}
\end{equation*}%
hence%
\begin{equation*}
\lambda _{j}^{h}\left\langle \Pi _{h}u_{k},u_{j}^{h}\right\rangle
_{I}=\lambda _{k}\left\langle u_{k},u_{j}^{h}\right\rangle _{I}.
\end{equation*}%
We subtract $\lambda _{k}\left\langle \Pi _{h}u_{k},u_{j}^{h}\right\rangle
_{I}$ from both sides above and we get%
\begin{equation*}
\left( \lambda _{j}^{h}-\lambda _{k}\right) \left\langle \Pi
_{h}u_{k},u_{j}^{h}\right\rangle _{I}=\lambda _{k}\left\langle u_{k}-\Pi
_{h}u_{k},u_{j}^{h}\right\rangle _{I},
\end{equation*}%
which in turn gives%
\begin{equation*}
\left\vert \left\langle \Pi _{h}u_{k},u_{j}^{h}\right\rangle _{I}\right\vert
\leq \rho _{k}^{h}\left\vert \left\langle u_{k}-\Pi
_{h}u_{k},u_{j}^{h}\right\rangle _{I}\right\vert .
\end{equation*}%
From (\ref{8.7}) we have%
\begin{equation}
\left\Vert \Pi _{h}u_{k}-w_{k}^{h}\right\Vert _{0,I}^{2}\leq \left( \rho
_{k}^{h}\right) ^{2}\sum_{j\neq k}\left\langle u_{k}-\Pi
_{h}u_{k},u_{j}^{h}\right\rangle _{I}^{2}\leq \left( \rho _{k}^{h}\right)
^{2}\left\Vert u_{k}-\Pi _{h}u_{k}\right\Vert _{0,I}^{2}.  \label{8.8}
\end{equation}%
To deal with the last term in (\ref{8.6}), we point out that if we establish%
\begin{equation}
\left\Vert u_{k}^{h}-w_{k}^{h}\right\Vert _{0,I}\leq \left\Vert
u_{k}-w_{k}^{h}\right\Vert _{0,I},  \label{8.9}
\end{equation}%
then%
\begin{equation}
\left\Vert u_{k}^{h}-w_{k}^{h}\right\Vert _{0,I}\leq \left\Vert u_{k}-\Pi
_{h}u_{k}\right\Vert _{0,I}+\left\Vert \Pi _{h}u_{k}-w_{k}^{h}\right\Vert
_{0,I},  \label{8.10}
\end{equation}%
with both terms on the right hand side above having been estimated. From (%
\ref{wk}) we have%
\begin{equation*}
u_{k}^{h}-w_{k}^{h}=u_{k}^{h}\left( 1-\left\langle \Pi
_{h}u_{k},u_{k}^{h}\right\rangle _{I}\right) .
\end{equation*}%
Also%
\begin{equation*}
\left\Vert u_{k}\right\Vert _{0,I}=\left\Vert u_{k}^{h}-w_{k}^{h}\right\Vert
_{0,I}\leq \left\Vert w_{k}^{h}\right\Vert _{0,I}\leq \left\Vert
u_{k}\right\Vert _{0,I}+\left\Vert u_{k}^{h}-w_{k}^{h}\right\Vert _{0,I}
\end{equation*}%
and since $u_{k},u_{k}^{h}$ are normalized, we have%
\begin{equation*}
1-\left\Vert u_{k}-w_{k}^{h}\right\Vert _{0,I}\leq \left\vert \left\langle
\Pi _{h}u_{k},u_{k}^{h}\right\rangle _{I}\right\vert \leq 1+\left\Vert
u_{k}-w_{k}^{h}\right\Vert _{0,I}
\end{equation*}%
from which we see that%
\begin{equation*}
\left\vert \left\vert \left\langle \Pi _{h}u_{k},u_{k}^{h}\right\rangle
_{I}\right\vert -1\right\vert \leq \left\Vert u_{k}-w_{k}^{h}\right\Vert
_{0,I}.
\end{equation*}%
By choosing 
\begin{equation*}
\left\langle \Pi _{h}u_{k},u_{k}^{h}\right\rangle _{I}\geq 0,
\end{equation*}%
we conclude that (\ref{8.9}) is satisfied. \ Utilizing (\ref{8.6}), (\ref%
{8.8}) and (\ref{8.10}) we conclude that there is an appropriate choice of
the sign of $u_{k}^{h}$ such that%
\begin{equation*}
\left\Vert u_{k}-u_{k}^{h}\right\Vert _{0,I}\leq 2\left( 1+\rho
_{k}^{h}\right) \left\Vert u_{k}-\Pi _{h}u_{k}\right\Vert _{0,I}\lesssim
C(k)h^{p+1}.
\end{equation*}%
To get the energy norm estimate we proceed as follows:%
\begin{eqnarray*}
\left\Vert u_{k}-u_{k}^{h}\right\Vert _{E,I}^{2} &\lesssim &{\mathcal{B}}%
_{\varepsilon }\left( u_{k}-u_{k}^{h},u_{k}-u_{k}^{h}\right) =\mathcal{B}%
_{\varepsilon }\left( u_{k},u_{k}\right) -2\mathcal{B}_{\varepsilon }\left(
u_{k},u_{k}^{h}\right) +\mathcal{B}_{\varepsilon }\left(
u_{k}^{h},u_{k}^{h}\right) \\
&=&\lambda _{k}-2\lambda _{k}\left\langle u_{k},u_{k}^{h}\right\rangle
_{I}+\lambda _{k}^{h}=\lambda _{k}\left[ 1-2\left\langle
u_{k},u_{k}^{h}\right\rangle _{I}\right] -\lambda _{k}+\lambda _{k}^{h} \\
&=&\lambda _{k}\Vert u_{k}-u_{k}^{h}\Vert _{0,I}^{2}+\lambda
_{k}^{h}-\lambda _{k} \\
&\lesssim &C(k)h^{2(p-1)}.
\end{eqnarray*}%
This completes the proof.
\end{proof}


\section{Numerical results}

\label{nr}

In this section we present the results of numerical computations for the
approximation of (\ref{de}) by cubic Hermite polynomials (i.e. $p=3$) in the
case when the data is chosen as $a(x)=e^{x},b(x)=x$. Since no exact solution
is available, we use a reference solution for the calculation of the errors
computed with higher accuracy. First we would like to verify the result of
Theorem \ref{thm_evals}, so in Figure \ref{fig1} we show the estimated
percentage relative error in the first two (smallest) eigenvalues, $%
100\times |\lambda _{i}-\lambda _{i}^{h}|/|\lambda _{i}|,i=1,2$ versus the
number of degrees of freedom $DOF$ (i.e. the dimension of the subspace) in a
log-log scale. We used $p=3$ and the resulting lines have slope $-4$ $%
(=-2p+2)$, just as Theorem \ref{thm_evals} predicts.

\begin{figure}[h]
\begin{center}
\includegraphics[width=0.45 \textwidth]{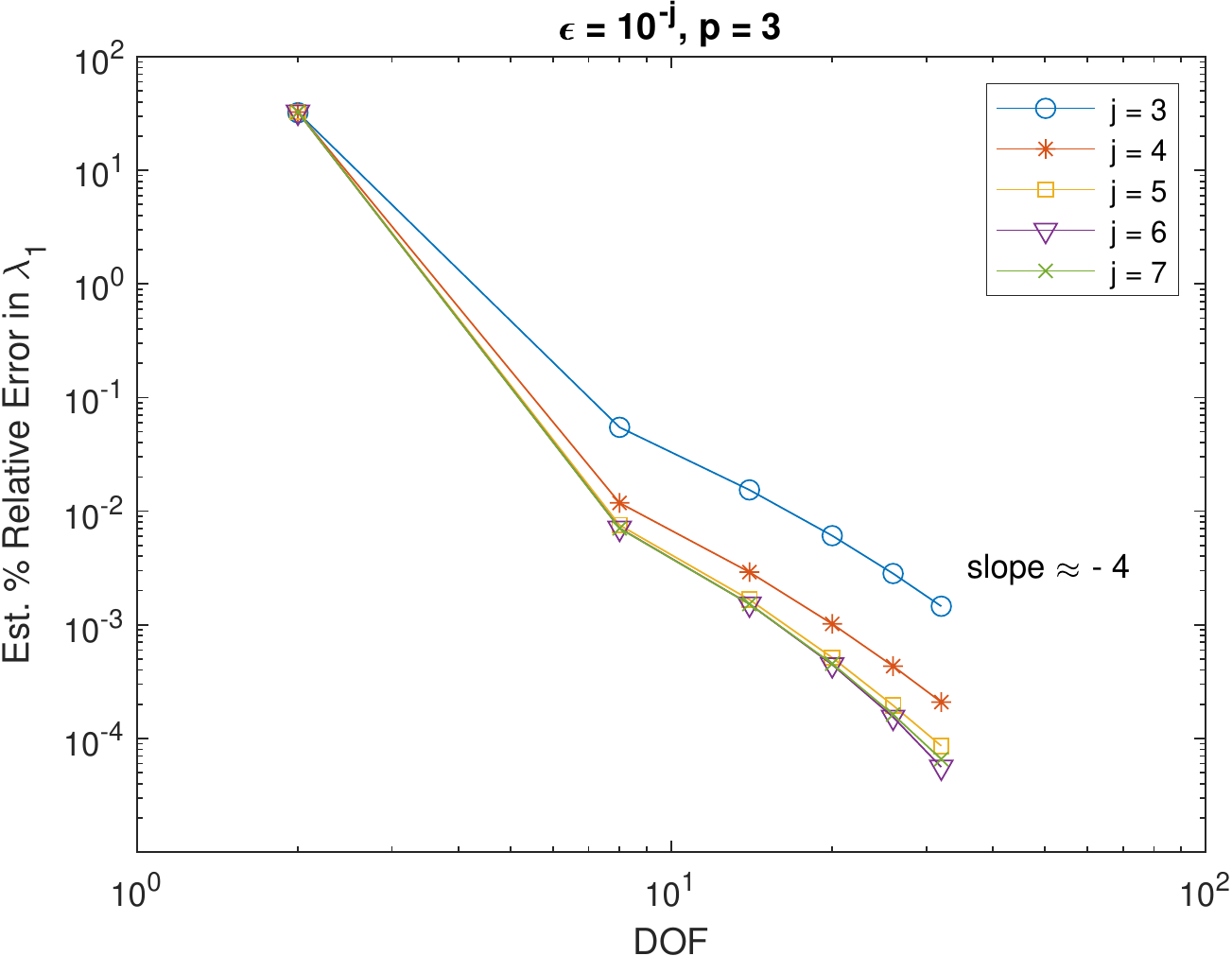} \mbox{} %
\includegraphics[width=0.43 \textwidth]{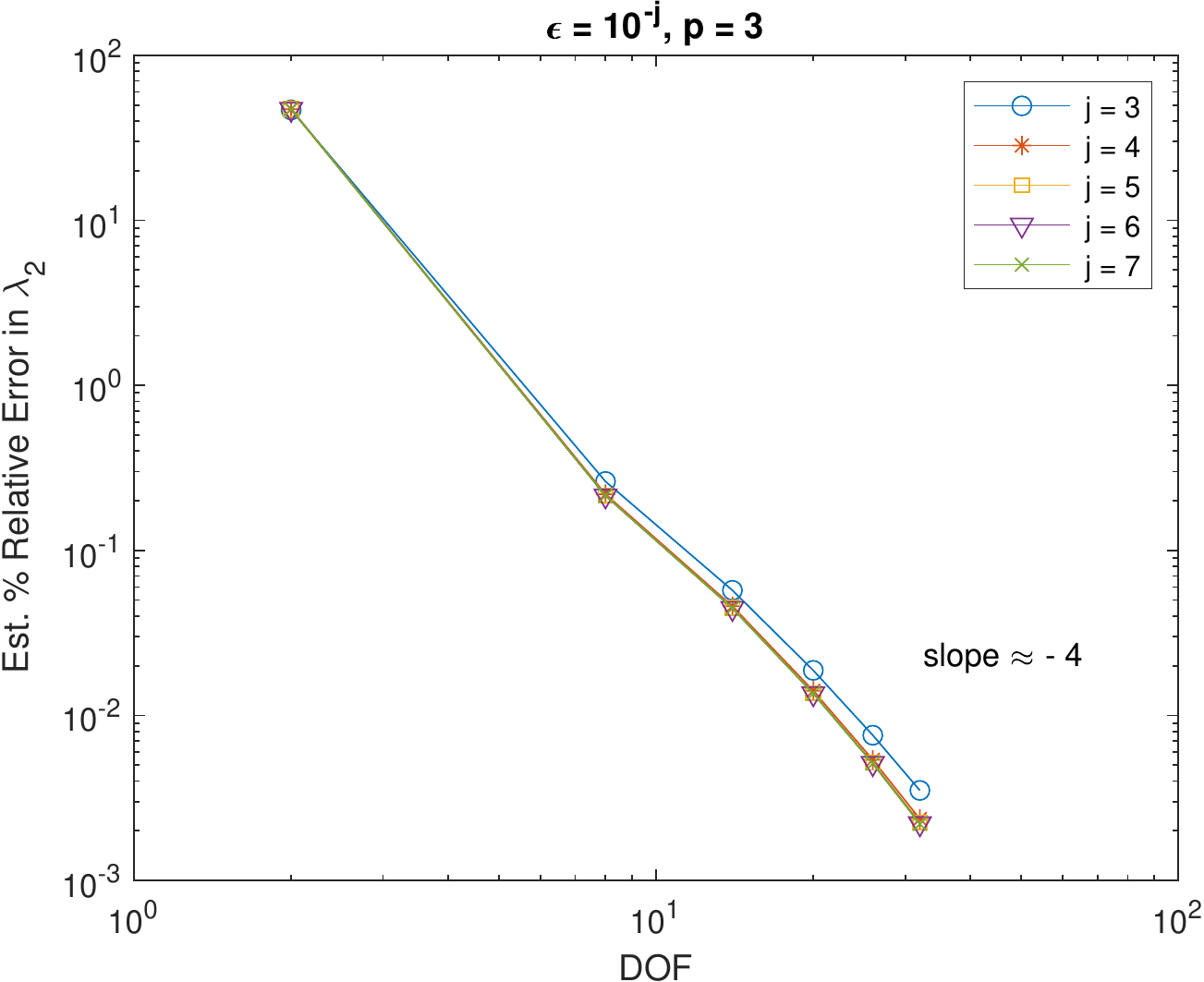}
\end{center}
\caption{Estimated convergence in $\protect\lambda_1$ (left) and $\protect%
\lambda_2$ (right).}
\label{fig1}
\end{figure}


We also show in Table \ref{table1} the computations for the first 5
eigenvalues, for $\varepsilon =10^{-6}$ (the same behavior was noticed for
other values of $\varepsilon $). We see that for larger eigenvalues the
convergence takes longer to set in, as was also observed for non-singularly
perturbed eigenvalue problems (see, e.g. \cite{Boffi}). In Figure \ref%
{compare} we illustrate this phenomenon, by comparing the convergence
between $\lambda _{1}$ and $\lambda _{5}$, for $\varepsilon =10^{-3},10^{-6}$%
. As can be seen, while $\varepsilon \rightarrow 0$ does not affect the
behavior (after all, the method is proven to be robust), there is a clear
difference between the case $\lambda _{1}$ and the case $\lambda _{5}$,
which suggests that the constants $C(k)$ in Theorem \ref{thm_evals}, grow
with $k$.


\begin{table}[h]
\begin{center}
\begin{tabular}{|c|c|c|c|c|c|c|c|}
\hline
$\lambda _{i}^{N}\backslash DOF$ & $2$ & $8$ & $14$ & $20$ & $26$ & $32$ & $%
38$ \\ \hline
$\lambda _{1}^{N}$ & $22.1093$ & $16.6812$ & $16.6803$ & $16.6801$ & $%
16.6801 $ & $16.6801$ & $16.6801$ \\ \hline
$\lambda _{2}^{N}$ & $94.9592$ & $64.6500$ & $64.5403$ & $64.5203$ & $%
64.5148 $ & $64.5130$ & $64.5122$ \\ \hline
$\lambda _{3}^{N}$ & $-$ & $145.7632$ & $144.7402$ & $144.3536$ & $144.2593$
& $144.2278$ & $144.2149$ \\ \hline
$\lambda _{4}^{N}$ & $-$ & $264.6963$ & $258.3972$ & $257.0769$ & $256.2126$
& $255.9574$ & $255.8615$ \\ \hline
$\lambda _{5}^{N}$ & $-$ & $423.2341$ & $410.7243$ & $402.9403$ & $401.7117$
& $400.1930$ & $399.6647$ \\ \hline
\end{tabular}%
\end{center}
\caption{Approximate eigenvalues for $\protect\varepsilon=10^{-6}$.}
\label{table1}
\end{table}


\begin{figure}[h]
\begin{center}
\includegraphics[width=0.55 \textwidth]{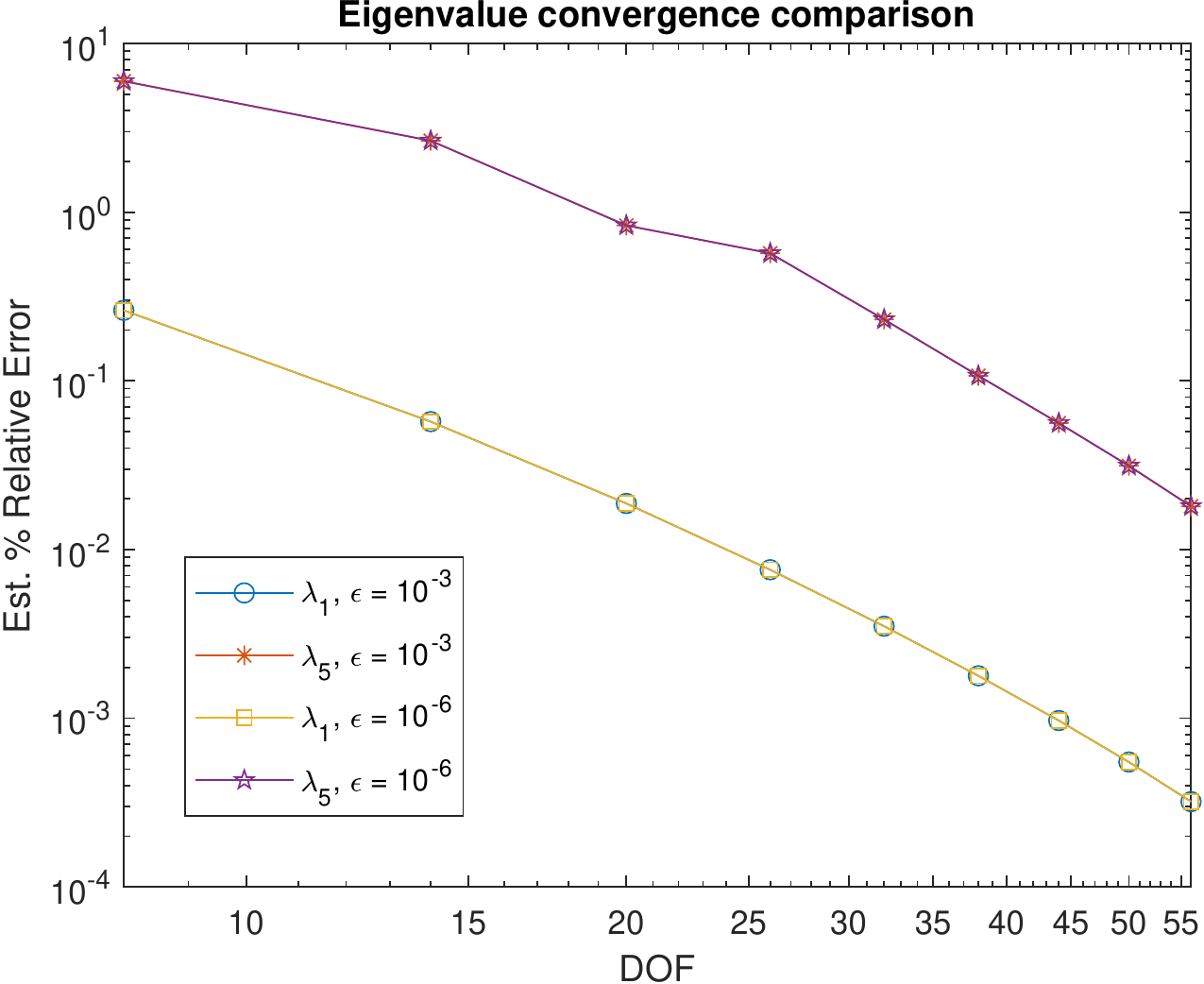}
\end{center}
\caption{Convergence comparison for $\protect\lambda_1$ and $\protect\lambda%
_5$.}
\label{compare}
\end{figure}


We now turn our attention to the eigenfunctions: Figure \ref{evcs} shows the
first two approximate eigenfunctions $u_{1}^{h},u_{2}^{h}$ and their
derivatives. The computations shown were performed for $\varepsilon =10^{-3}$
and $p=3$ with $N=32$ nodal points. We see the boundary layers being present
in the derivatives and how the proposed method is able to capture them.


\begin{figure}[h]
\begin{center}
\includegraphics[width=0.4 \textwidth]{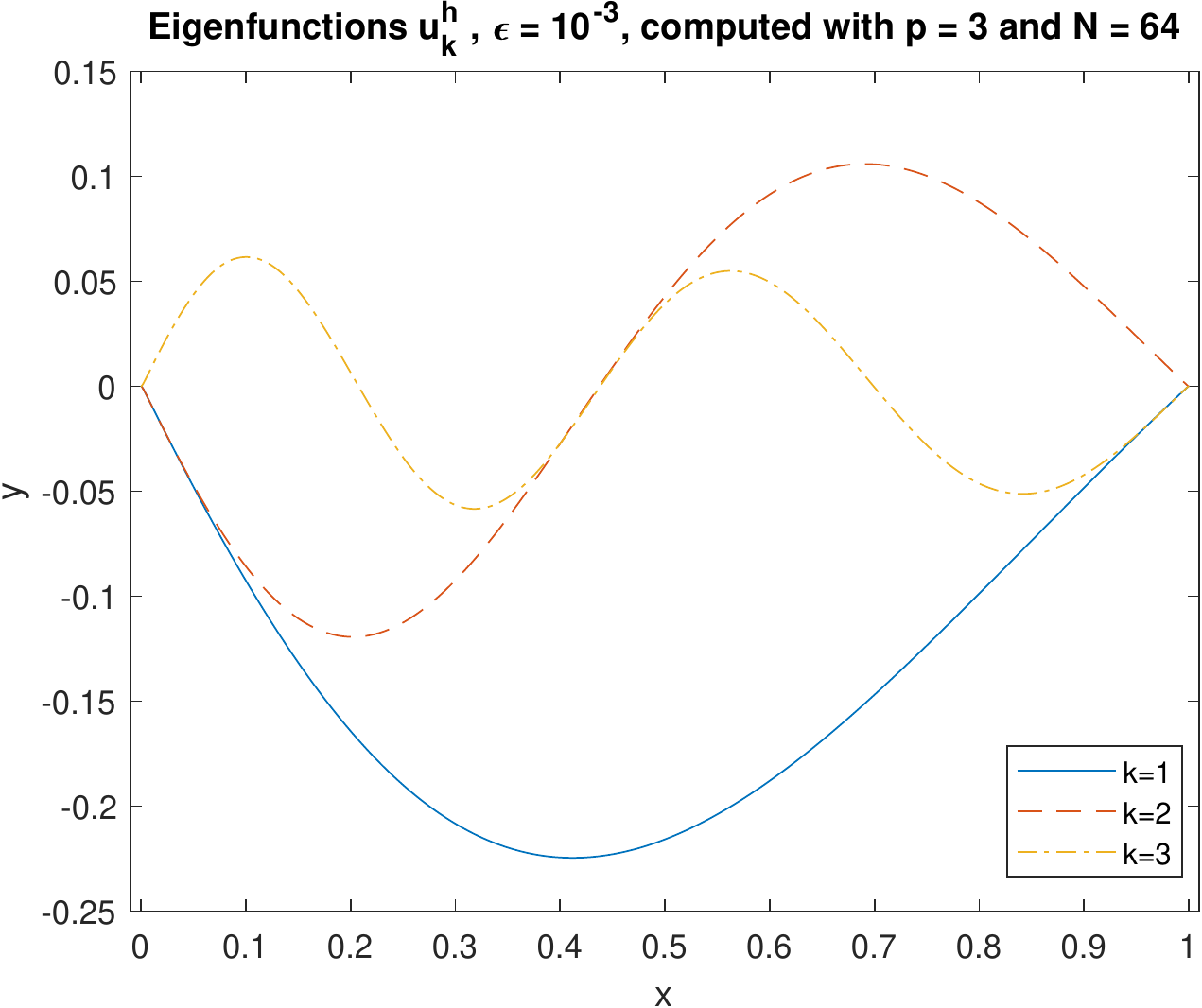} \mbox{} %
\includegraphics[width=0.4 \textwidth]{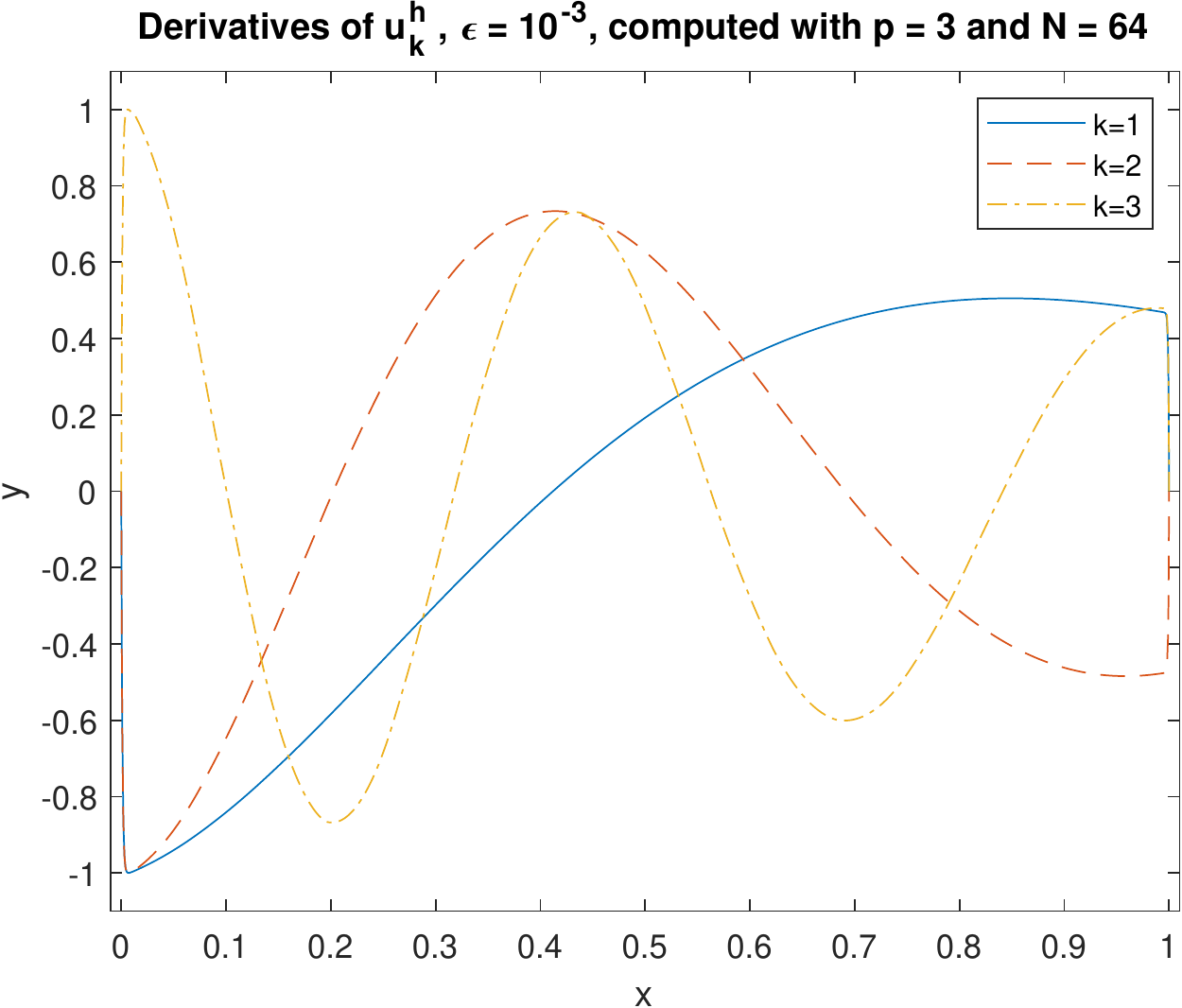}
\end{center}
\caption{Approximate eigenvectors (left) and their derivatives (right) for $%
\protect\varepsilon=10^{-3}$.}
\label{evcs}
\end{figure}


In terms of convergence, we compute the percentage relative error in the
energy norm 
\begin{equation*}
Error=100\times \frac{\left\Vert u_{i}-u_{i}^{h}\right\Vert _{E,I}}{%
\left\Vert u_{i}\right\Vert _{E,I}},
\end{equation*}%
and plot it versus the number of $DOF$, in a log-log scale. We do so for $%
i=1 $ and show the result in Figure \ref{evc1}. The slope is approximately $%
-2(=p-1)$, which verifies the prediction of Theorem \ref{thm_evecs}.


\begin{figure}[h]
\begin{center}
\includegraphics[width=0.55 \textwidth]{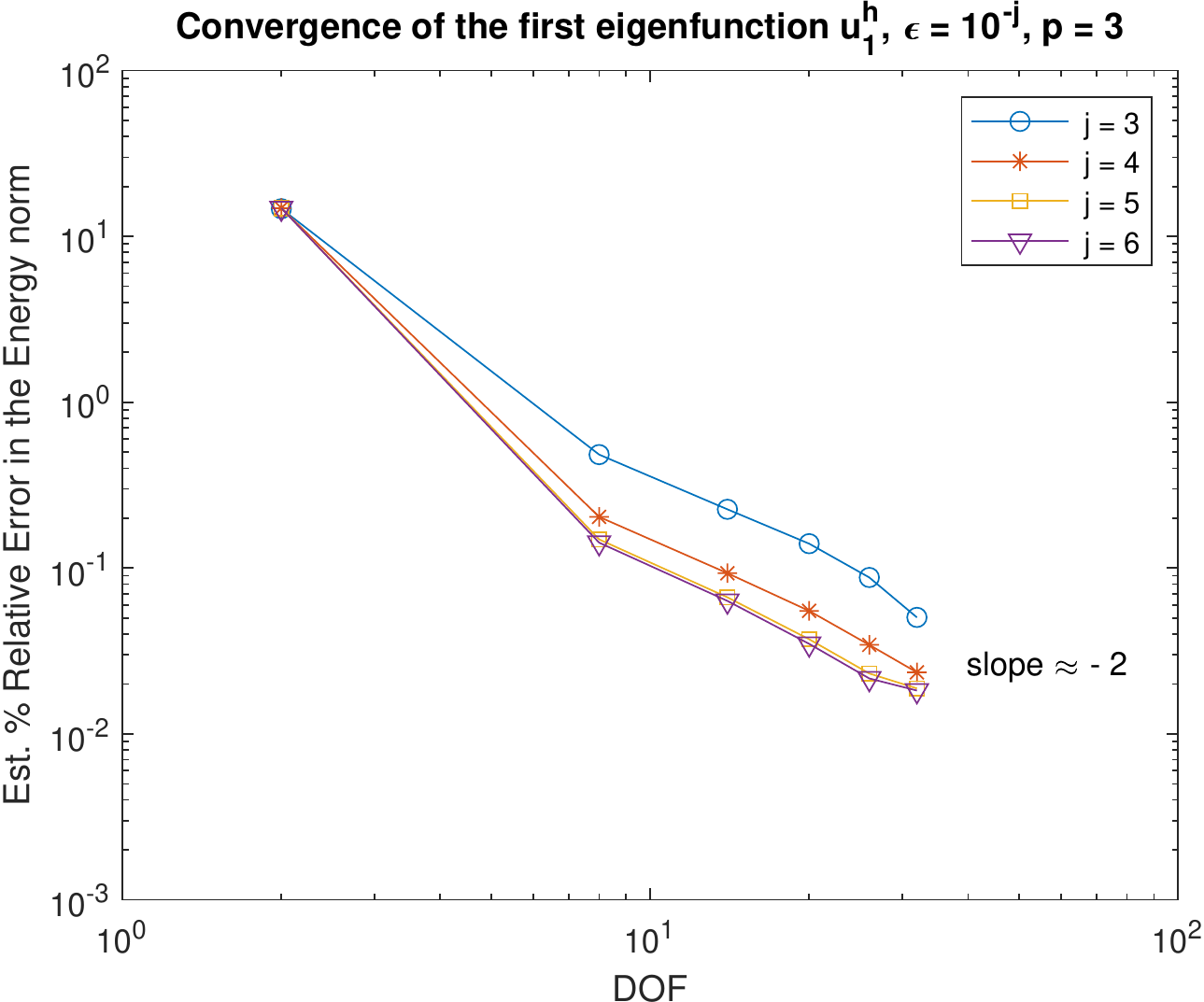}
\end{center}
\caption{Energy norm convergence for the fist eigenfunction.}
\label{evc1}
\end{figure}


We also consider the error in the first eigenfunction and its derivative
measured in a `discrete maximum norm', defined as 
\begin{equation*}
error=100\times \frac{\max_{x_{\ell }\in \lbrack 0,1]}\left\vert \left\vert
u_{1}(x_{\ell })\right\vert -\left\vert u_{1}^{h}(x_{\ell })\right\vert
\right\vert }{\left\vert u_{1}(x_{\ell })\right\vert }.
\end{equation*}%
The points $x_{\ell }\in \lbrack 0,1]$ are chosen so that we have equal
number of points in the layer regions and outside -- we used 1000 point in
each. This is not covered by our theory, so it may be seen as an extension
of our results. Figure \ref{evc2} shows the convergence rate which seems to
be robust and of order $O\left( h^{p}\right) $ for the eigenvector and $%
O\left( h^{p-1}\right) $ for its derivative.


\begin{figure}[h]
\begin{center}
\includegraphics[width=0.4 \textwidth]{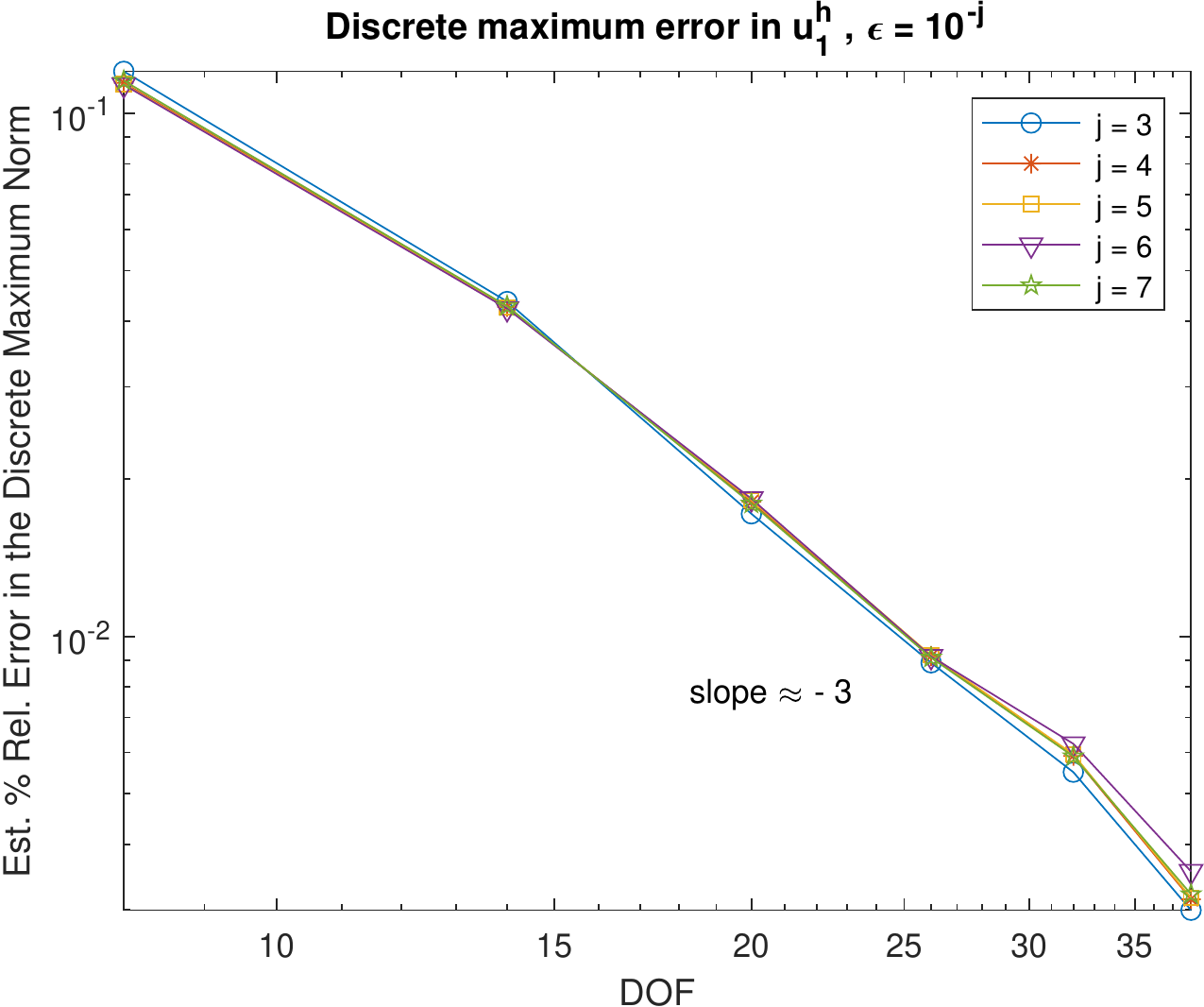} \mbox{} %
\includegraphics[width=0.4 \textwidth]{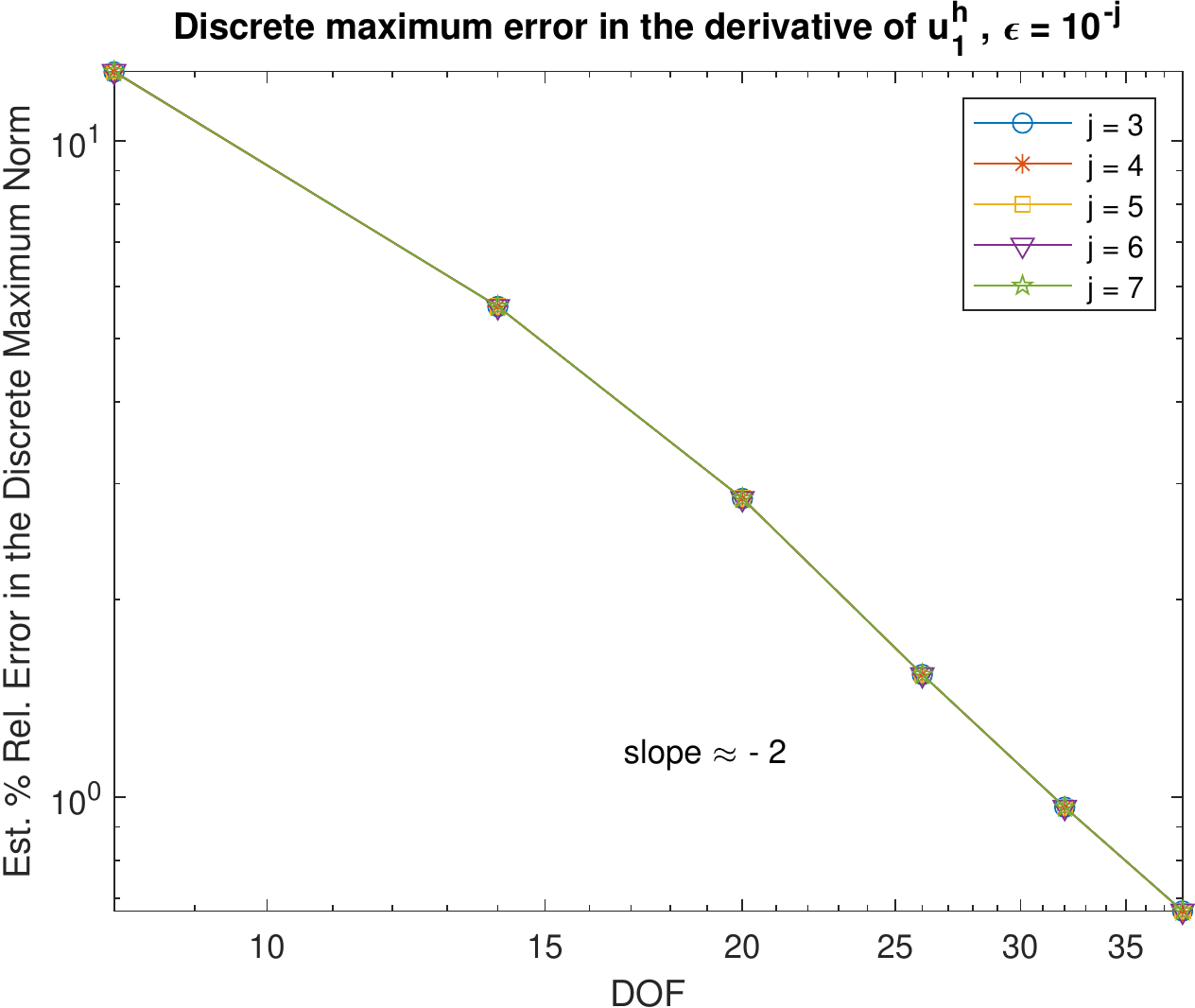}
\end{center}
\caption{Discrete maximum norm convergence for the fist eigenfunction (left)
and its first derivative (right).}
\label{evc2}
\end{figure}


\section{Conclusions}

\label{concl} We considered a singularly perturbed fourth order eigenvalue
problem and the numerical approximation of its solution using the $h$%
-version FEM with Hermite polynomials of degree $p\geq 3$ defined on an
exponentially graded mesh. We established optimal, uniform (in $\varepsilon $%
) convergence for both the eigenvalues and the eigenfunctions, when the
error was measured in absolute value and in the energy norm, respectively.
We should point out that a smallness assumption on $h$ is necessary to
establish our results and this is seen in our numerical experiments,
especially for higher modes. While the analysis was performed in
one-dimension, the results are extendable to higher dimensions, since the
boundary layer effect is one-dimensional (in the direction normal to the
boundary). Unfortunately, constructing $C^{1}$ elements in two-dimensions is
difficult -- even for simple domains. Some progress has been made \cite{WXL}%
, but we believe that a mixed formulation is a viable alternative choice.
This is the focus of our current research efforts.


\end{document}